\documentclass{amsart}
\usepackage[utf8]{inputenc}
\usepackage{amssymb, amsmath, amsthm, url}
\usepackage{hyperref}
\usepackage{todonotes}

\newtheorem{theorem}{Theorem}[section]
\newtheorem{lemma}[theorem]{Lemma}

\newtheorem{proposition}[theorem]{Proposition}
\newtheorem*{proposition-intro}{Proposition}
\newtheorem{corollary}[theorem]{Corollary}

\theoremstyle{definition}

\theoremstyle{remark}
\newtheorem{remark}[theorem]{Remark}

\newcommand{\bbP}{\ensuremath{\mathbb{P}}}

\newcommand{\bbZ}{\ensuremath{\mathbb{Z}}}

\newcommand{\bbQ}{\ensuremath{\mathbb{Q}}}
\newcommand{\bbR}{\ensuremath{\mathbb{R}}}

\newcommand{\ord}{\ensuremath{\operatorname {ord}}}
\newcommand{\vol}{\ensuremath{\operatorname {vol}}}

\newcommand{\GL}{\ensuremath{\operatorname {GL}}}

\newcommand{\Y}{\ensuremath{Y_{\bullet}}}

\title{On the number of vertices of Newton--Okounkov polygons}
\author{Joaquim Roé and Tomasz Szemberg}
\date{}

\begin{document}
\maketitle

\begin{abstract}
	The Newton--Okounkov body of a big divisor $D$ on a smooth surface is a numerical invariant in the form of a convex polygon.
	We study the geometric significance of the \emph{shape} 
	of Newton--Okounkov polygons of ample divisors, showing that they share several important properties of Newton polygons on toric surfaces.
	In concrete terms, sides of the polygon are associated to some particular irreducible curves, and their lengths are determined by the intersection numbers of these curves with $D$.
	
	As a consequence of our description we determine the numbers $k$ such that $D$ admits some $k$-gon as a Newton--Okounkov body, elucidating the relationship of these numbers with the Picard number of the surface, which was first hinted at by work of Küronya, Lozovanu and Maclean.\\
\end{abstract}

\section{Introduction}

The Newton--Okounkov body of a line bundle with respect to an admissible flag is defined as follows (see \cite{LM09},\cite{KK12},\cite{Bou14}).
Let $S$ be a normal projective variety of dimension $d$ (the case we shall deal with in the paper is that of a surface, $d=2$), let $D$ be a big divisor class on $S$, and fix a flag
\[\Y: S=Y_0 \supset Y_1  \supset \dots \supset Y_d= \{pt\}\]
which is \emph{admissible}, i.e., $Y_i$ is an irreducible subvariety of codimension $i$, smooth at the point $Y_d$, for each $i$.
Let $g_i$ be a local equation for $Y_i$ in $Y_{i-1}$ around the point $Y_d$.
Then $\Y$ determines a rank $d$ valuation $v_{\Y}$ on the field of rational functions of $S$, namely $v_{\Y}(f)=(v_1(f),\dots,v_d(f))$, where $v_i$ are defined recursively setting $f_1=f$ and
\begin{align*}
v_i(f)&=\ord_{Y_i}(f_i),& i&=1,\dots,d,\\
f_{i+1}&=(f_i / g_i^{v_i(f)} )|_{Y_i},& i&=1,\dots,d-1.
\end{align*}
By trivializing $\mathcal{O}_S(D)$ in an arbitrary neighborhood of $Y_d$, the valuation $v_{\Y}$ may be applied to global sections of multiples $\mathcal{O}_S(kD)$, and the Newton--Okounkov body of $D$ with respect to $\Y$ is the convex body
\[
\Delta_{\Y}(D)=\overline{\left\{\left.\frac{v_{\Y}(s)}{k}\right| s\in H^0(S,\mathcal{O}_S(kD))\right\}}.
\]

Whereas the \emph{volume} of the Newton--Okounkov body is well known to equal $\vol(D)/(\dim S)!$ for every $\Y$, its \emph{shape}, and in particular its dependence on $\Y$, is still an intriguing subject.

There are two situations in which the work of Lazarsfeld--Musta\c{t}\u{a} \cite{LM09} allows to prove that $\Delta_{\Y}(D)$ is a polytope.
If $S$ is a toric variety, $D$ is a torus-invariant divisor, and the flag $Y$ is composed of torus-invariant subvarieties, then \cite{LM09} proves that $\Delta_{\Y}(D)$ is (up to the action of $\GL_n(\bbZ)$) just the Newton polytope associated to $D$ in toric geometry (see \cite[\S 3.4]{Ful93}). 
If $S$ is a surface, on the other hand, then Küronya--Lozovanu--Maclean \cite{KLM12,KL18} used the description of \cite{LM09} to show that for every $D$ and $\Y$, $\Delta_{\Y}(D)$ is a polygon.
A close analysis of their construction reveals that the shape of general Newton--Okounkov polygons (on surfaces) reflects the geometry of the pair $(D,\Y)$ much like Newton polygons do in the toric case.
Indeed, when $S, D$ and $\Y$ are toric,  $\Delta_{\Y}(D)$ is a polygon with vertices in $\bbZ^2$, with a side corresponding to each prime toric divisor (if $D$ is ample), the selfintersections of these prime divisors determine the slopes of these sides, and their intersection with $D$ equals the lattice length of the corresponding sides.
On an arbitrary smooth projective surface, associated to each pair $(D, \Y)$, there is a configuration of irreducible curves  playing the role of the torus-invariant prime divisors: each side of $\Delta_{\Y}(D)$ corresponds to one or more of these irreducible curves, their slopes are rational and determined by the intersection matrix of the configuration, and their lengths are determined by the intersection numbers of $D$ with the curves in the configuration. This amounts to \cite[Theorem B]{KLM12}, which we prefer to state here in the following form (we elaborate on the details in section \ref{sec:slopes-lengths}).

\begin{proposition-intro}
	Let $S$ be a normal projective surface, $D$ a big divisor on $S$, and $\Y:S\supset C\supset \{p\}$ an admissible flag.
	Let $\mu$ be the maximal real number such that $D-\mu C$ is pseudo-effective, and let $D-\mu C=P + N$ be the Zariski decomposition of $D-\mu C$.
	
	The intersection matrix of $C$ and of the irreducible components of $N$ determine all possible slopes of sides of $\Delta_{\Y}(D)$.
	For a fixed $D$ and $C$, in the set of all big divisors $D'$ such that the negative part $N'$ of the Zariski decomposition of $D'-\mu_{D'}C$ has the same support as $N$,
	and for each possible slope, the length of the corresponding side is a function of the intersection numbers of $D'$ with these curves (where we understand that if this length is zero then no side with that slope exists).
	
	Moreover, the lower sides are related to connected components of $N$ passing through $p$ whereas the upper sides are related to connected components of $N$ intersecting $C$ at other points.
\end{proposition-intro}

It was observed in \cite{KL18} that the number of vertices of $\Delta_{\Y}(D)$ is bounded above by $2\rho(S)+2$, where $\rho(S)$ denotes the Picard number. We show that the slightly stronger bound  $2\rho(S)+1$ holds and is sharp, i.e., for any given natural number $\rho$ there are surfaces $S$ with $\rho(S)=\rho$, ample divisors $D$ and flags $\Y:S\supset C\supset \{p\}$ such that $\Delta_{\Y}(D)$ is a $(2\rho+1)$-gon. We also determine, in terms of configurations of negative curves on a given smooth surface $S$, the numbers $k$ for which there is a flag $\Y$ such that $\Delta_{\Y}(D)$ is a $k$-gon.

The role of Zariski decompositions in the determination of $\Delta_{\Y}(D)$ provide a strong relationship with Zariski chamber decompositions, and in fact the subdivision of the interval $[0,\mu]$ given by the projections of the sides of the polygon is also induced by the Zariski walls crossed by the ray that starts from $D$ in the direction of $-C$. Somewhat surprisingly for the authors, the numbers $k$ for which $\Delta_{\Y}(D)$ can have $k$ sides (or the chambers that can be traversed by a ray emanating from $D$) are independent of $D$ as long as it is ample.

\medskip

In order to state our main result, we introduce some new invariants attached to a configuration of negative curves.
For an effective divisor $N=C_1+\dots+C_k$ with negative definite intersection matrix, consider the following two numbers:
\begin{itemize}
	\item $mc(N)$ denotes the largest number of irreducible components of a connected divisor contained in $N$.
	\item $mv(N)$ denotes $k+mc(N)+4$ if $k<\rho-1$, and $k+mc(N)+3$ if $k=\rho-1$.
\end{itemize}

Given a smooth projective surface $S$, let $mv(S)=\max\{mv(N) \,|\, N=C_1+\dots+C_k\text{ negative definite}\}.$
Our main result is the following:

\begin{theorem}\label{main}
	On every smooth projective surface $S$, and for every big divisor $D$,
	\[\max_{\Y}\left\{\#\operatorname{vertices}(\Delta_{\Y}(D))\right\}\le mv(S)\]
(where the maximum is taken over all admissible flags $Y$).
If $D$ is ample, then for every $3\le v\le mv(S)$ there exists a flag $\Y$ such that $\Delta_{\Y}(D)$ has exactly $v$ vertices.
\end{theorem}

Note that by the Hodge index theorem, $mv(S)$ is defined and bounded above by $2\rho+1$; we also show that this upper bound is sharp:

\begin{corollary}
	Given a positive integer $\rho$, there is a projective smooth surface $S$ with Picard number $\rho(S)=\rho$, a divisor $D$ and a flag $\Y$ such that $\Delta_{\Y}(D)$ has $2\rho+1$ vertices.
\end{corollary}

The deep analogy of Newton--Okounkov polygons with the Newton polygons in toric geometry is the departure point for our work. In the case of finitely generated valuation semigroup, the connection with toric geometry goes far beyond an analogy and has found important applications via \emph{toric degenerations} (see e.g. \cite{HK}). 
When the pseudo-effective cone of $S$ is finitely generated, only finitely many negative divisors $N=C_1+\dots+C_k$ exist, with the $C_i$ being generators, and the application of Theorem \ref{main} becomes especially straightforward. 
In these cases, a concrete description of the polygon $\Delta_{\Y}(D)$ is already available, as a Minkowski sum of triangles and line segments (at least if $S$ is del Pezzo or the flag $\Y$ is general enough, see \cite{LSS14}, \cite{PS14}).

In the absence of finite generation, the meaning of the toric analogy is far from being well understood, and in \cite{KLM12} it appears implicitly only.
In section \ref{sec:lm-klm} we recall the description of Newton--Okounkov polygons uncovered by \cite[Section 6.2]{LM09} and \cite{KLM12} from the point of view outlined above, which we then use to study these polygons, with special emphasis on the number of vertices (or sides) they possess.
The analysis of these boundaries, done in sections \ref{sec:interior} and \ref{sec:right}, allows to determine the number of lower and upper vertices; 	Proposition \ref{inner} is the technical key to all results in this paper.
It is worth stressing that, analogously to the boundaries of classical Newton polygons, the lower boundary encodes the local behavior of $D$ near $p$, whereas the upper boundary encodes behaviour ``at infinity''.
In section \ref{sec:proofs} we prove Theorem \ref{main} and, finally, in Propositions \ref{pro:slopes} and \ref{pro:lengths} we show how to explicitly determine the slopes and lengths (so the whole shape) of Newton--Okounkov polygons from intersection numbers as indicated above.
To construct flags $\Y$ that give polygons with the desired number of points we use Lemma \ref{ordered-negative} which may be interesting in itself: it shows that the Zariski chambers that can be crossed by a ray starting from an ample class in the Néron--Severi space are independent of the particular class chosen.

\paragraph{Acknowledgements}
We thank the anonymous referees for carefully reading our manuscript and pointing out an inaccuracy in the original statement and proof of Lemma \ref{ordered-negative}.
We are grateful to  Alex Küronya, Julio Moyano-Fernández and Matthias Nickel for helpful discussions. We gratefully acknowledge partial  support from the Spanish MinECO Grants No. MTM2016-75980-P and PID2020-116542GB-I00, and by National Science Centre grant 2018/ 30/M/ST1/00148.

\section{Newton--Okounkov polygons}
\label{sec:lm-klm}

In this section we recall the description of Newton--Okounkov polygons on surfaces given by Lazarsfeld--Musta\c{t}\u{a} \cite{LM09} and Küronya--Lozovanu--Maclean \cite{KLM12}.
Given a surface $S$, we denote $NS(S)$ its Néron--Severi group (i.e., the group of divisors modulo numerical equivalence), a finitely generated abelian group of rank $\rho(S)$.
When needed, we will consider $\bbQ$-divisors and $\bbR$-divisors with the conventions of \cite{Laz04I}, and $NS(S)_\bbR$ will be the space of numerical classes of $\bbR$-divisors endowed with the bilinear form given by the intersection product.

Fix a smooth surface $S$, a big divisor $D$ and an admissible flag $\Y:S\supset C\supset \{p\}$ on $S$.
For every real number $t$, consider the $\bbR$-divisor $D_t=D-tC$ and, if $D_t$ is effective or pseudo-effective, denote its Zariski decomposition
\[D_t=P_t+N_t.\]
Let $\nu=\nu_C(D)$ be the coefficient of $C$ in the negative part $N_0$ of the Zariski decomposition of $D$, and $\mu=\mu_C(D)=\max\{t\in\bbR|D_t\text{ is pseudo-effective}\}$.
Note that $D_\mu$ belongs to the boundary of the pseudo-effective cone, in particular it is not big (as big classes form the interior of the pseudo-effective cone).
For every $t\in[\nu,\mu]$ define $\alpha(t)=(N_t\cdot C)_p$, i.e., the local intersection multiplicity of the negative part of $D_t$ and $C$ at $p$, and $\beta(t)=\alpha(t)+P_t\cdot C$.
Lazarsfeld and Musta\c{t}\u{a} showed in \cite[Section 6.2]{LM09} that $\Delta_{\Y}(D)$ is the region in the plane $(t,s)$ defined by the inequalities $\nu\le t \le\mu$, $\alpha(t)\le s\le\beta(t)$.
Note that
$\alpha$ and $\beta$ are continuous piecewise linear functions in the interval $[\nu,\mu]$, respectively convex and concave.

Observing that $N_t$ increases with $t$ and looking at $N_\mu$, Küronya--Lozovanu--Maclean proved in \cite{KLM12} that $\alpha$ is nondecreasing, that the values $t\in (\nu,\mu)$ where $\alpha$ or $\beta$ fails to be linear are exactly those where $D_t$ crosses walls between Zariski chambers \cite{BKS04}, and that there are finitely many such crossed walls, in fact at most as many as components in $N_\mu$.

\begin{remark}\label{beta-intpts}
	Let $q_1,\dots,q_r$ be the intersection points of $N_\mu$ and $C$ different from $p$. It follows immediately from the description above that $\beta(t)= D\cdot C - t\, C^2 - \sum_{i=1}^{r} (N_t\cdot C)_{q_i}$.
\end{remark}

Our approach to understanding the number of vertices in Newton--Ok\-ounkov polygons is to further analyze the dependence of $N_t$ on $t$, and from this derive information on the functions $\alpha$ and $\beta$. 
So, let us briefly recall the proof of polygonality of $\Delta_{\Y}(D)$ due to \cite{LM09} and \cite{KLM12}. Call $C_1,\dots,C_n$ the irreducible components of $N_\mu$, numbered \emph{in order of appearance} in the support of $N_t$, that is, denoting
$$t_i=\inf\{t\in[\nu,\mu]\,|\,C_i \text{ in the support of }N_t\}$$
for each $i \in\{1,\dots,n\}$, one has $t_0:=\nu\le t_1 \le \dots\le t_n< t_{n+1}:=\mu$.
We can write
\[N_t=a_1(t)C_1+\dots+a_n(t)C_n,\]
where $a_i(t)$ are (continuous) functions $[\nu,\mu]\to \bbR$.
%
The equations defining the Zariski decomposition $D_t=P_t+N_t$ tell us that, for $t\in[t_{i-1},t_{i}]$ and $1\le j < i$, $P_t\cdot C_j=0$, or equivalently $N_t\cdot C_j=D_t\cdot C_j$. Therefore $a_j(t)$ are solutions of the linear system of equations
\begin{equation}\label{eq:linear-equations}
\begin{array}{r@{}lc}
(a_1(t)C_1+\dots+a_{i-1}(t)C_{i-1})\cdot C_j&{}=(D-tC)\cdot C_j, & 1\le j< i,\\
a_j(t)&{}=0, & i\le j\le n. 
\end{array}\end{equation}
These solutions are unique because the intersection matrix $(C_k\cdot C_j)_{1\le k,j< i}$ is nonsingular. 
Since the independent terms $(D-tC)\cdot C_j$ are affine linear functions of $t$, so are the solutions $a_j(t)=a_{j0}+a_{j1}t$, i.e., $a_j$ is affine linear on each interval $[t_{i-1},t_i]$. 
It follows then that $\alpha$ and $\beta$ are continuous affine linear on each interval $[t_{i-1},t_i]$ (which can be degenerate, if $t_i=t_{i-1}$) so $\Delta_{\Y}(D)$ is a polygon and the  first coordinate of every vertex equals one of the $t_i$, $i\in\{0,\dots,n+1\}$.

\section{Interior vertices}
\label{sec:interior}

We keep the notations of the previous section, namely $\Y:S\supset C\supset \{p\}$ is an admissible flag, $D_t=D-tC=P_t+N_t$, and $\nu=\nu_C(D)$, $\mu=\mu_C(D)$, $\alpha$, $\beta$, $t_i$ are as above.

Vertices $P=(t_i,s)$ of $\Delta_{\Y}(D)$ can be classified as \emph{leftmost} (if $t_i=t_0=\nu$), \emph{rightmost} (if $t_i=t_{n+1}=\mu$) and \emph{interior} (if $\nu<t_i<\mu$).
A vertex $P$ is also called \emph{upper} if $s=\beta(t_i)$ and \emph{lower} if $s=\alpha(t_i)$.
Before proceeding to the determination of the $t_i$ for which $\Delta_{\Y}(D)$ has upper and lower interior vertices, we recall a result on relative negative parts of Zariski decompositions, essentially due to Zariski:

\begin{lemma}\label{relative-bauer}
	Let $D$ be an effective divisor on a smooth surface, let $D=P+N$ be its Zariski decomposition, and let $N=a_1C_1+\dots+a_nC_n$, $a_i\in \bbQ$ be the decomposition into irreducible components. For every subset $I\subset\{1,\dots,n\}$, let $b_i,i\in I$, be the solutions to the system of linear equations
	\begin{equation}
	\left(D-\sum_{i\in I} b_i C_i\right)\cdot C_j=0, \quad j\in I.
	\end{equation}
	Then $b_i\le a_i$ for each $i\in I$.
\end{lemma}
\begin{proof}
	First observe that we may assume $I\subsetneq\{1,\dots,n\}$, as otherwise $b_i=a_i$ and there is nothing to prove.
	The presentation of Zariski decomposition given in \cite{BCK12} 
	in terms of linear algebra will immediately yield that there is $J\subset\{1,\dots,n\}$, $I\subsetneq J$, such that the solutions $b_i'$ to the corresponding system of equations
	\begin{equation}
	\left(D-\sum_{i\in J} b'_i C_i\right)\cdot C_j=0, \quad j\in J,
	\end{equation}
	satisfy $b_i\le b_i'$ for each $i\in I$, which applied recursively gives what we need.
	Indeed, denote $\mathbf{p}=[D-\sum_{i\in I} b_i C_i]$, $\mathbf{v}=[D]$, and $\mathbf{e}_i=[C_i], 1\le i\le n$ considered as vectors in $NS(S)_{\bbR}$.
	The space $\langle \mathbf{e}_i\rangle_{i\in I}$ is, in the language of \cite{BCK12}, a special negative definite subspace of the support space of $\mathbf{v}$.
	So, the hypotheses of \cite[Lemma 5.3]{BCK12} are satisfied, and therefore $D'=D-\sum_{i\in I} b_i C_i$ is effective.
	Since $I\subsetneq \{1,\dots,n\}$, there is at least one curve $C_{j}$ among $C_1,\dots, C_n$ such that $D'\cdot C_{j}<0$ (otherwise the Zariski decomposition of $D$ would not involve all $C_1,\dots, C_n$); let $J$ be $I\cup\{j \,|\,D'\cdot C_j<0\}$.
	Since the intersection form on $\langle\mathbf{e_i}\rangle_{i\in J}$ is negative definite, there is a unique $\mathbf{n}=\sum_{i\in J} a_i\mathbf{e}_i$ with $\mathbf{n}\cdot \mathbf{e}_i=\mathbf{p}\cdot\mathbf{e}_i \forall i\in J$.
	Then \cite[Lemmas 5.2 \& 5.3]{BCK12} give that $\mathbf{n}$ and $\mathbf{p}-\mathbf{n}$ are effective.
	The latter effectiveness gives $b_i\le b_i'$ for each $i\in I$, as wanted.
\end{proof}

Fix a pair of indices $1\le i \le k\le n$ such that $t_{i-1}<t_i=\dots=t_k<t_{k+1}$.
This means that $C_i, \dots, C_k$ are the components of the negative part of all $N_{t_i+\varepsilon}$ with $\epsilon>0$ that are not components of $N_{t_i}$.
Write, for $j=1,\dots, k$
\[
a_j(t_i+\epsilon)=\begin{cases}
a_{j0}+a_{j1} \,\varepsilon  & \text{ if }-1\ll \varepsilon\le 0,\\
a_{j0}+a_{j1}^+ \,\varepsilon & \text{ if }0<\varepsilon\ll1,\\
\end{cases}
\]
where $a_{j0}=0$ and $a_{j1}=0$ for $j\ge i$, and $a_{j1}^+>0$ for every $j \le k$.

\begin{lemma}\label{intersect-slope}
	For every $j=1,\dots, k$, the inequality $a_{j1}^+\ge a_{j1}$ holds. If $C_{j}\cdot C_{j'}>0$ and $a_{j1}^+> a_{j1}$ then $a_{j'1}^+> a_{j'1}$.
\end{lemma}
\begin{proof}
	By definition, $(D_{t_i+\varepsilon}-\sum_{j=1}^{k}(a_{j1}\varepsilon+a_{j0})C_j)\cdot C_{j'}=0$ for every $\varepsilon$ and every $j'<i$.
	Therefore, by Lemma \ref{relative-bauer}, for every $0<\varepsilon\ll1$ and every $j$, $a_{j1}^+\varepsilon +a_{j0}\ge a_{j1}\varepsilon+a_{j0}$, whence $a_{j1}^+\ge a_{j1}$.

	For the second claim, we only need to take care of the case $j'<i$.
	Then, we have for every $\varepsilon$
	\begin{align}
	(D_{t_i+\varepsilon}-\sum_{j=1}^{i-1}(a_{j1}\varepsilon+a_{j0})C_j)\cdot C_{j'}=0, \label{left-zd}\\
	(D_{t_i+\varepsilon}-\sum_{j=1}^{k}(a_{j1}^+\varepsilon+a_{j0})C_j)\cdot C_{j'}=0. \label{right-zd}
	\end{align}
	Subtracting both equalities, it results
	\begin{align*}
	\sum_{j=1}^{k}(a_{j1}^+-a_{j1}) C_j\cdot C_{j'}=0, \text{ i.e.,}\\
	\sum_{j\ne j'}(a_{j1}^+-a_{j1}) C_j\cdot C_{j'}=(a_{j'1}-a_{j'1}^+)C_{j'}^2.
	\end{align*}
	All terms on the left hand side of the last equality are nonnegative, and if $C_j\cdot C_{j'}>0$ and $a_{j1}^+>a_{j1}$ then at least one of them is positive, so the right hand side must be positive and $a_{j'1}^+>a_{j'1}$, as claimed.
\end{proof}

\begin{proposition}\label{inner}
Let as above $C_1, \dots, C_n$ be the irreducible components of $N_\mu$, $t_i=\inf\{t \,|\,C_i \text{ in }N_t\}$ with $\nu=t_0 \le t_1\le \dots \le t_{m+1}=\mu$.
Fix some $i,k$ with $t_{i-1}<t_i=\dots=t_k<t_{k+1}$.

The Newton--Okounkov body $\Delta_{\Y}(D)$ has an interior lower vertex with first coordinate $t_i$ if and only if, for every $\varepsilon>0$, there is a connected component of $N_{t_i+\varepsilon}$ that goes through $p$ and contains at least one of the $C_{i},\dots,C_k$.
It has an interior upper vertex with first coordinate $t_i$ if and only if, for every $\varepsilon>0$, there is a connected component of $N_{t_i+\varepsilon}$ that intersects $C$ at a point different from $p$ and contains at least one of the $C_{i},\dots,C_k$ .

\end{proposition}
\begin{proof}
	As above, $C_1, \dots, C_k$ are the irreducible components of $N_{t_i+\varepsilon}$ for $0<\varepsilon\ll1$, and
	\[
	N_{t_i+\varepsilon}=
	\sum_{j=1}^{k}(a_{j1}^+ \varepsilon+a_{j0})C_j\text{ if }0<\varepsilon\ll1.
	\]
	Because of the description above of the lower boudary $\alpha(t)$ of $\Delta_{\Y}(D)$, it is clear that if no component $C_i$ passes through $p$ then $\alpha(t)=0$ for all $t\in (\nu,t_i+\varepsilon)$ for small $\varepsilon$, and there is no lower vertex with first coordinate $t_i$.
	So assume some component passes through $p$, and let $J\subset \{1,\dots k\}$ be such that $\bigcup_{j\in J} C_j$ is the connected component of $N_{t_i+\varepsilon}$ that contains $p$. We have
	\[ \alpha(t_i+\varepsilon)=\begin{cases}
	\sum_{j\in J}(a_{j1} \varepsilon+a_{j0})(C_j\cdot C)_p & \text{ if }-1\ll \varepsilon\le 0,\\
	\sum_{j\in J}(a_{j1}^+ \varepsilon+a_{j0})(C_j\cdot C)_p & \text{ if }0<\varepsilon\ll1,
	\end{cases}
	\]
	so there is a lower vertex with first coordinate $t_i$ if and only if $a_{j1}^+>a_{j1}$ for some $j$ such that $C_j$ passes through $p$.
	By Lemma \ref{intersect-slope} this certainly happens if there is some component $C_j$, $j\in J$, $j\ge i$. Conversely, if all components $C_j$, $j\in J$ have $j<i$, the equations involving $a_{j1}$ and $a_{j1}^+$ in \eqref{left-zd} and \eqref{right-zd} are equal, so $a_{j1}=a_{j1}^+$ for all $i\in J$, and there is no interior lower vertex with first coordinate $t_i$.	
	
	The proof of the second claim is entirely analogous, and we will be brief.
	Because of Remark \ref{beta-intpts}, if no component $C_j$ meets $C$ at a point different from $p$ then $\alpha(t)=D\cdot C - t \, C^2$ for all $t\in (\nu,t_i+\varepsilon)$ for small $\varepsilon$, and there is no upper vertex with first coordinate $t_i$.
	On the other hand, if some component $C_j$ does meet $C$ at a point different from $p$, by Remark \ref{beta-intpts},
	there is an upper vertex with first coordinate $t_i$ if and only if $a_{j1}^+>a_{j1}$ for some $j$ such that $C_j$ meets $C$ at a point different from $p$, and Lemma \ref{intersect-slope} finishes the proof just as in the case of lower vertices.	
\end{proof}

\begin{corollary}\label{bound-interior}
	The number of interior lower (resp. upper) vertices of $\Delta_{\Y}(D)$ is bounded above by the number of irreducible components of the connected component $N(p)$ of $N_\mu$ that meets $C$ at $p$ (resp. the number of irreducible components of $N_\mu$ in some $N(q)$ for some $q\in C\setminus\{p\}$).
\end{corollary}

\section{Rightmost vertices}
\label{sec:right}

Keep the notation of the previous section, namely $D$ is a big divisor, $C$ an irreducible curve, $p$ a point on $C$, $D_t=D-tC$, $\mu=\max\{t\,|\,D_t \text{ pseudo-effective}\}$, $\nu= \nu_C(D)$, $D_t=P_t+N_t$ the Zariski decomposition for $\nu\le t\le \mu$,  and the irreducible components of $N_\mu$ are $C_1,\dots,C_n$.

\begin{lemma}\label{lem-rightmost}
	The subspace
	\[V=\langle[D_\nu],[C_1],\dots,[C_n]\rangle \subset NS(S)_\bbR\]
	has dimension $n+1$ and the intersection form restricted to $V$ is nondegenerate with signature $(1,n)$.
\end{lemma}
\begin{proof}
	The negative part $N_\nu$ of the Zariski decomposition $D_\nu=N_\nu+P_\nu$ (if nonzero) satisfies $N_\nu\le N_\mu$ and hence is a combination of the $C_i$.
	Therefore $[P_\nu]\in V$.
	Since $D$ is big, $P_\nu$ is big and nef and therefore $P_\nu^2>0$.
	As the intersection matrix of the $C_i$ is negative definite, and $[P_\nu]\in V$, it follows that $\dim V=n+1$ and there is some class $[P]\in V$ orthogonal to all $C_i$, which moreover has $P^2>0$. 	
\end{proof}
\begin{proposition} \label{rightmost}
	If the numerical equivalence class $[C]$ belongs to the subspace
	\[V=\langle[D],[C_1],\dots,[C_n]\rangle \subset NS(S)_\bbR\]
	then $\Delta_{\Y}(D)$ has $1$ rightmost vertex.
	If $[C]$ is ample and does not belong to $V$ then $\Delta_{\Y}(D)$ has $2$ rightmost vertices.
\end{proposition}

\begin{proof}
	By the previous lemma there is a divisor $P$ orthogonal to all $C_i$, with $P^2>0$ and such that $V=\langle[P],[C_1],\dots,[C_n]\rangle$.

	If $[C]\in V$ then $[P_\mu]=[D]-\mu [C]-[N_\mu]\in V$, and since $P_\mu$ is orthogonal to all $C_i$, it must be $[P_\mu]=a[P]$ for some $a\in \bbR$.
	If $a\ne 0$, then $P_{\mu}^2=a^2P^2>0$ and $P_\mu$ would be big, contradicting the definition of $\mu$, so $a=0$ and $P_\mu=0$.
	Hence $\alpha(\mu)-\beta(\mu)=P_\mu\cdot C=0$, which means that $\Delta_Y(D)$ has a single rightmost vertex. 	
	
	On the other hand, if $[C]\notin V$ then $[P_\mu]=[D]-\mu [C]-[N_\mu]\notin V$, and in particular $[P_\mu]\ne0$.
	If moreover $[C]$ is ample, and so belongs to the \emph{interior} of the nef cone, then its intersection with every nonzero class on the (dual) pseudo-effective cone is positive.
	Therefore $\alpha(\mu)-\beta(\mu)=P_\mu\cdot C>0$, which means that $\Delta_Y(D)$ has two rightmost vertices.
\end{proof}

\begin{corollary} \label{one-rightmost}
	If the negative part $N_\mu$ of the Zariski decomposition $D_\mu=P_\mu+N_\mu$ has $\rho(S)-1$ irreducible components, then the polygon $\Delta_{\Y}(D)$ has exactly one rightmost vertex.
\end{corollary}
\begin{proof}
	By lemma \ref{lem-rightmost}, the subspace
	\[V=\langle[D],[C_1],\dots,[C_n]\rangle \subset NS(S)_\bbR\]
	has dimension equal to $\rho(S)-1+1$ and is therefore equal to the whole Néron--Severi space.
	The claim then follows from Proposition \ref{rightmost}.	
\end{proof}

\section{Counting vertices}
\label{sec:proofs}
Recall from the introduction that for an effective divisor $N=C_1+\dots+C_n$ with negative definite intersection matrix, $mc(N)$ denotes the largest number of irreducible components of a connected divisor contained in $N$,
\[mv(N)=\begin{cases}
n+mc(N)+4 &\text{ if }n<\rho(S)-1,\\
n+mc(N)+3 &\text{ if }n=\rho(S)-1,
\end{cases}
\]
and $mv(S)$ is the maximum of all $mv(N)$ for $N=C_1+\dots+C_n$ negative definite.

First we prove that $mv(S)$ is an upper bound for the number of vertices of every Newton--Okounkov body on $S$, and then we give a constructive proof that for every ample divisor class $D$, every number of vertices allowed by the bound is realized by some flag.

\begin{theorem}\label{thm:bound}
	On every smooth projective surface $S$, for every big divisor $D$ and every flag $\Y$, \(\#\operatorname{vertices}(\Delta_{\Y}(D))\le mv(S).\)
\end{theorem}
\begin{proof}
	We shall be more precise, showing that if  $\nu=\nu_C(D)$, $\mu=\mu_C(D)$ and $N=N_\mu=a_1C_1+\dots+a_n C_n$ is the negative part of the Zariski decomposition of $D_\mu=D-\mu C$, then the number of vertices is bounded by $mv(N)$.
	By Proposition \ref{inner}, the number of upper iterior vertices is bounded by $n$, and the number of lower interior vertices is bounded by  $mc(N)$; the number of leftmost and rightmost vertices is always at most 2, but if $n=\rho(S)-1$, then by Corollary \ref{one-rightmost} there is exactly 1 rightmost vertex, and the bound follows.
\end{proof}

\begin{corollary}
	Let $S$ be a smooth projective algebraic surface, $D$ a big divisor and $Y=\{C, p\}$ an admissible flag on $S$.
	The polygon $\Delta_{\Y}(D)$ has at most $2\,\rho(S)+1$ vertices.
\end{corollary}

For the construction of flags leading to bodies with the desired number of vertices we shall need the following lemma, which may be of independent interest.

\begin{lemma}\label{ordered-negative}
	Let $N=C_1+\dots+C_k$ be an effective divisor with negative definite intersection matrix (admitting $k=0$ in which case $N=0$), and $A$ an ample divisor. There is an irreducible curve $C$ whose class is ample, such that for every $t$ with $A-tC$ pseudo-effective, the negative part of its Zariski decomposition is supported on $N$, and moreover, for	every $i=1,\dots, k$,
	\begin{enumerate}
		\item $C$ intersects $C_i$ in at least two points,
		\item denoting $N_t$ the negative part of the Zariski decomposition of $A-tC$,
		$\sup\{t\in \bbQ \,|\, C_i \text{ is not contained in }N_t\}$ is a finite positive real number $t_i$, and
		\item $t_1< \dots < t_k$.
	\end{enumerate}
	Moreover, the numerical class of the curve $C$ can be taken in the span $\langle A, C_1, \dots, C_k\rangle$.
\end{lemma}
Note that the second condition simply means that the support of $N_t$ is exactly $N$ for $t$ large enough.
Observe that the lemma still holds when $N=0$, as the claims in that case are empty.

\begin{proof}

	We will prove by induction on $k$ that there are positive rational numbers $a_1, \dots, a_k$ such that $B=A-a_1C_1-\dots-a_kC_k$ is ample, and that every irreducible curve $C\in|mB|$ satisfies the last two desired properties.
	Since for every ample class $B$ there is a multiple $mB$ and an irreducible curve $C\in|mB|$ that intersects each $C_i$ in at least two points, we shall be done.
	
	If $k=1$, choose a positive integer $a$ such that the divisor class $B=A-(1/a)C_1$ is ample. Then for every $C\in |mB|$, $A_{1/m}=A-(1/m)C=(m/a)C_1=N_{1/m}$, so that
	$$0<\sup\{t\in \bbQ \,|\, C_1 \text{ is not contained in }N_t\}<1/m,$$
	and we are done.
	
	Now assume the claim is true for the divisor $C_1+\dots+C_{k-1}$, and let $a_1,\dots,a_{k-1}$ be positive rational numbers such that  $B'=A-a_1C_1-\dots-a_{k-1}C_{k-1}$ is ample and satisfies the two conditions
	\begin{enumerate}
		\item denoting $N_t'$ the negative part of the Zariski decomposition of $A-tB'$,
		$\sup\{t\in \bbQ \,|\, C_i \text{ is not contained in }N_t'\}$ is a finite positive real number $t_i'$,  and
		\item $t_1'< \dots < t_{k-1}'$.
	\end{enumerate}
	Of course this implies the two analogous conditions for $A-tC$ for every $C\in |mB'|$.
	Note that for every $t\in[0,1/m]$, since $A_t=(1-mt)A+mt(a_1C_1+\dots+a_{k-1}C_{k-1})$, with $(1-mt)A$ nef and $a_1C_1+\dots+a_{k-1}C_{k-1}$ effective, by the extremality properties of the Zariski decomposition it follows that $N_t\le tm(a_1C_1+\dots+a_{k-1}C_{k-1})$ (with equality if and only if $t=1/m$). In particular, all components of $N_t$ are among the $C_i$.
	
	Choose rational numbers $s_i$ with $0=s_0<t_1'<s_1<t_2'<\dots<s_{k-2}<t_{k-1}'<s_{k-1}<1$.
	The choices made guarantee that the irreducible components of $N'_{s_i}$ are exactly $C_1, \dots, C_i$, and, since for every $i<j\le k$ we have $N'_{t_j}\ge N_{s_i}$ and $P'_{t_j}\cdot C_j=0$, it follows that 
	 $P'_{s_i}\cdot C_j\ge (t'_j-s_i)B'\cdot C_j>0$ for all $i<j\le k$.
	 Therefore, by continuity of the Zariski decomposition (see \cite[Proposition 1.14]{BKS04}, there exist $\varepsilon_i>0$ such that for all $0<a_k\le \varepsilon_i$, the irreducible components of the negative part in the Zariski decomposition of $A-s_i(B'-a_kC_k)$ are also exactly $C_1, \dots, C_i$.
	Thus it suffices to choose a rational $a_k$ smaller than $\varepsilon_0, \dots, \varepsilon_{k-1}$ and set $B=A-a_1C_1-\dots-a_kC_k$, because clearly $N_1=a_1C_1+\dots+a_kC_k$ and therefore
	$$t_{k-1}<s_{k-1}<\sup\{t\in \bbQ \,|\, C_k \text{ is not contained in }N_t\}<1,$$
	completing the induction step.
	
	The class $B$ is by construction a combination of $A$ and the $C_i$, so the class of $C$ belongs to $\langle A, C_1, \dots, C_k\rangle$ as claimed.
\end{proof}

\begin{lemma}\label{ordered-negative-maximal}
	Let $N=C_1+\dots+C_k$ be a maximal effective divisor with negative definite intersection matrix, i.e., such that there exists no curve $C'$ distinct from $C_1, \dots, C_k$ 
	with $N+C'$ having negative definite intersection matrix, and let $A$ be an ample divisor. If $k<\rho(S)-1$, there is an irreducible curve $C$ satisfying all properties of Lemma \ref{ordered-negative} but whose numerical class is linearly independent of $\langle A, C_1, \dots, C_k\rangle$.
\end{lemma}
\begin{proof}
	Consider the class $B$ from the proof of Lemma \ref{ordered-negative}; we can slightly modify $B$ to obtain a $B''$ which still satisfies the properties and whose numerical class is independent, as follows.
	Assuming $t_i=\sup\{t\in \bbQ \,|\, C_k \text{ is not contained in }N_t\}$ for $i=1,\dots, k$ as above, choose rational numbers $s_i \in (t_i,t_{i+1})$ and $s_k\in (t_k,1)$.
	Let $Z$ be an irreducible curve whose numerical class is independent of those of $C$ and the $C_i$.
	Then $P_{s_i}\ge (1-s_i)A$, so $P_{s_i}\cdot Z>0$, and by continuity of the Zariski decomposition, there exist $\varepsilon_i>0$ such that for all $|b|\le \varepsilon_i$ the Zariski decomposition of $A-s_i(B+bZ)$ has exactly the components $C_1, \dots, C_i$ in its negative part. 
	The fact that $N$ is maximal guarantees that for no $t>s_k$ any other negative curve appears in $N_t$.
	Then the desired class is $B''=B+bZ$ for some $|b|\le \varepsilon_i$ for every $i$.	
\end{proof}
\begin{theorem}
	On every smooth projective surface $S$, for every ample divisor $A$ and every integer $v$, $3\le v\le mv(S)$, there exists a flag $\Y$ such that 	\(\#\operatorname{vertices}(\Delta_{\Y}(A))=v.\)
\end{theorem}
\begin{proof}
	Choose an effective divisor $N_{mv}=C_1+\dots+C_k$ with negative definite intersection matrix, such that $mv(S)=mv(N_{mv})$, and assume that its components have been ordered in such a way that for every $1\le i \le mc(N_{mv})$, the divisor $C_1+\dots+C_i$ is connected.
	By the definition of $mv(N)$, it is not restrictive to assume that $N_{mv}$ is maximal, i.e., there exists no curve $C'$ with $N_{mv}+C'$ having negative definite intersection matrix.
	
	If $k<\rho(S)-1$, for every $i\le j\le k$, $mv(C_1+\dots+C_j)=mv(N_{mv})-k+j$, and for $0\le j \le i$, $mv(C_1+\dots+C_j)=mv(N_{mv})-k-i+2j$.
	On the other hand, if $k=\rho(S)-1$, for every $i\le j< k$, $mv(C_1+\dots+C_j)=mv(N_{mv})-k+j+1$, and for $0\le j \le i$, $mv(C_1+\dots+C_j)=mv(N_{mv})-k-i+2j+1$.
	In any event,
	\[\{3,\dots,mv(S)\} \subset \bigcup_{N\le N_{mv}}\{mv(N)-1, mv(N)-2\} \cup\{mv(N_{mv})\}. \]
	Therefore, it will be enough to prove that, for every $N$ with negative definite intersection matrix:
	\begin{itemize}
		\item If $N$ is maximal, there is a flag $\Y$ such that $\Delta_{\Y}(A)$ has $mv(N)$ vertices.
		\item If $N$ is nonzero or has less than $\rho(s)-1$ components, there is a flag $\Y$ such that $\Delta_{\Y}(A)$ has $mv(N)-1$ vertices.
		\item If $N$ is nonzero and has less than $\rho(s)-1$ components, there is a flag $\Y$ such that $\Delta_{\Y}(A)$ has $mv(N)-2$ vertices.	
	\end{itemize}

	In the case of a maximal $N$ with less than $\rho(S)-1$ components, choose an irreducible curve $C$ satisfying the conditions of Lemma \ref{ordered-negative-maximal}, and let $p$ be one of the intersection points of $C$ and $C_1$ (unless $N=0$ in which case we choose an arbitrary $p\in C$).
	We claim that $A$, $\Y:S\supset C \supset\{p\}$ give a body with $mv(N)$ vertices.
	On the one hand, since $A$ is ample, $P_0=A$ and $P_0\cdot C>0$, so $\nu=0$ and $\Delta_{\Y}(A)$ has two leftmost vertices.
	Moreover, Proposition \ref{inner} ensures that $\Delta_{\Y}(A)$ has two interior vertices with first coordinate equal to the number $t_i$ given by Lemma \ref{ordered-negative-maximal} for $i=1,\dots,mc(N)$, whereas it only has an upper interior vertex for $mc(N)<i\le k$.
	Finally, as the numerical class of $C$ is independent of those of $A, C_1, \dots, C_k$, by Proposition \ref{rightmost} $\Delta_{\Y}(A)$ has two rightmost vertices.
	So, the total number of vertices is $mv(N)$.
	
	Now choose $C$ verifying the condtions of Lemma \ref{ordered-negative}, so that the class of $C$ belongs to the span $\langle A, C_1, \dots, C_k\rangle$. The shape of $\Delta_{\Y}(A)$ is as before, but with a single rightmost vertex; if $N$ has $\rho(S)-1$ components (in particular $N$ is maximal) the total number of vertices is $mv(N)$, otherwise it is $mv(N)-1$.
	
	Finally, if $N$ is nonzero we can pick $p$ differently, while keeping the same curve $C$ that satisfies the condtions of Lemma \ref{ordered-negative}. If $mc(N)=1$ we let $p$ be a point of $C$ not on $N$, and if $mc(N)>1$ then we take $p$ to be one of the intersection points of $C$ with $C_2$. In this way we obtain one lower point less, so if $N$ has $\rho(S)-1$ components the total number of vertices is $mv(N)-1$, otherwise it is $mv(N)-2$.
\end{proof}

\begin{corollary}
	Given integers $v,\rho$ with $3\le v \le 2\,\rho +1$, there exist a smooth projective algebraic surface $S$ with Picard number $\rho(S)=\rho$, a big divisor $D$ and an admissible flag $Y=\{C, p\}$ on $S$, such that $\Delta_{\Y}(D)$ has $v$ vertices.
\end{corollary}
\begin{proof}
	For $\rho=1$ there is nothing to prove, since Theorem \ref{thm:bound} shows that every big divisor and every admissible flag on a surface with Picard number 1 give rise to a triangular Newton--Okounkov body.
	
	So assume $\rho\ge 2$ and pick a surface $S_0$ with Picard number 1.
	Construct $S$ by successively blowing up points $p_1,\dots,p_{\rho-1}$ where $p_1\in S$ and for $i>1$, $p_i$ is a point on the exceptional divisor of the previous blowup.
	Then the exceptional divisor of the composition $S\rightarrow S_0$ is a connected divisor $N$ with $\rho-1$ components and negative definite intersection matrix, and hence $mv(S)=2\rho+1$.
\end{proof}
Note that this construction can be made starting from $S_0=\bbP^2$ and selecting each $p_i$ to lie in the strict transform of a fixed line; in that case the resulting surface $S$ is toric (and the Newton--Okounkov polygons obtained by toric flags have $\rho+2$ vertices, well short of the $2\rho+1$ vertices that are attainable with our construction).

\section{Slopes and lengths of sides}\label{sec:slopes-lengths}
Continue with the notation from previous sections, namely $D$ is a big divisor, $C$ is an irreducible curve, $p$ is a point on $C$, $\nu=\nu_C(D)$, $D_t=D-tC$,  $\mu=\max\{t\,|\,D_t \text{ pseudo-effective}\}$, $D_t=P_t+N_t$ is the Zariski decomposition for $0\le t\le \mu$, $C_1, \dots, C_n$ are the irreducible components in order of appearance, $t_i=\inf\{t\in[\nu,\mu]\,|\,C_i \text{ in the support of }N_t\}$, and $N_t=\sum a_i(t) C_i$.

\begin{proposition}\label{pro:slopes}
	The slopes of the sides of $\Delta_{\Y}(D)$ are determined by the intersection numbers $C\cdot C_j$ and $C_i\cdot C_j$ and the local intersection numbers $(C\cdot C_j)_p$.
\end{proposition}

\begin{proof}
	Recall from Section \ref{sec:lm-klm} that in the interval $[t_{i-1},t_i]$, $a_j(t)$ can be written as $a_j(t)=a_{j0}+a_{j1}(t)$ 
	satisfying equations \eqref{eq:linear-equations}.
	By looking at the coefficients of $t$ in \eqref{eq:linear-equations}, we see that
\begin{equation}
\begin{array}{r@{}lc}
a_{11}C_1\cdot C_j+\dots+a_{i-1,1}C_{i-1}\cdot C_j&{}=-C\cdot C_j, & 1\le j< i,\\
a_{j1}&{}=0, & i\le j\le n. 
\end{array}\end{equation}
	Thus the coefficients $a_{j1}$ are determined by the intersection numbers $C\cdot C_j$ and $C_i\cdot C_j$, and
	in the interval $[t_{i-1},t_i]$ we have
	$$
	\alpha(t)=(N_t\cdot C)_p=\sum_{j=1}^{i-1} (a_{j0}+a_{j1}t)(C_j\cdot C)_p, 
	$$
	i.e., the slope of the corresponding lower side of $\Delta_{\Y}(D)$ is
	$\sum a_{j1} (C_j\cdot C)_p$, which is determined by the intersection numbers $C\cdot C_j$ and $C_i\cdot C_j$ and the local intersection numbers $(C\cdot C_j)_p$.
	
	On the other hand, 
	\begin{gather*}
	\beta(t)=(N_t\cdot C)_p+P_t \cdot C=\\
	\sum_{j=1}^{i-1} (a_{j0}+a_{j1}t)(C_j\cdot C)_p+D\cdot C-\sum_{j=1}^{i-1} (a_{j0}+a_{j1}t)C_j\cdot C, 
	\end{gather*}
	i.e., the slope of the corresponding upper side of $\Delta_{\Y}(D)$ is
	$$\sum a_{j1} \left((C_j\cdot C)_p-C_j\cdot C\right),$$
	 which is determined by the intersection numbers $C\cdot C_j$ and $C_i\cdot C_j$ and the local intersection numbers $(C\cdot C_j)_p$.
\end{proof}

For a fixed flag $\Y:S\supset C \supset \{p\}$ and a fixed negative definite configuration $C_1,\dots,C_n$, let $\mathcal{D}_C(C_1,\dots, C_n)$ stand for the set of big divisors $D$ such that the negative components of $D-\mu C$ are exactly the $C_i$ numbered by order of appearance in $D-tC$.
Proposition \ref{pro:slopes} shows that the bodies $\Delta_{\Y}(D)$ for $D\in\mathcal{D}_C(C_1,\dots, C_n)$ share the overall shape (number of vertices and slopes of sides), differing only in the lengths of their sides. These lengths are determined by intersection numbers as follows.

\begin{proposition}\label{pro:lengths}
	For fixed $\Y:S\supset C \supset \{p\}$, and $C_1,\dots,C_n$,
	for every $D \in \mathcal{D}_C(C_1,\dots, C_n)$, the lengths of the sides of $\Delta_{\Y}(D)$ are determined by $\mu$  and the intersection numbers $D\cdot C_j$, and $D\cdot C$.	
\end{proposition}

Note that to actually compute the lengths of the sides of $\Delta_{\Y}(D)$, the intersection numbers appearing in Proposition \ref{pro:slopes} are needed too. The claim in \ref{pro:lengths} can be rephrased saying that the map 
\begin{align*}
\mathcal{D}_C(C_1,\dots, C_n) & \rightarrow \bbR^m \\
D &\mapsto (\text{lengths of sides})
\end{align*}
factors through
\begin{align*}
\mathcal{D}_C(C_1,\dots, C_n) & \rightarrow \bbR^{n+1} \\
D &\mapsto (D\cdot C_1,\dots,D\cdot C_n,D\cdot C).
\end{align*}

\begin{proof}
	The length of the leftmost vertical side is $P_\nu\cdot C=P_0\cdot C$; since the coefficients of $N_0=\nu C+a_1 C_1+ \dots +a_k C_k$ (where $k$ is the maximum index with $t_k=\nu$) are determined by:
\begin{align}\label{eq:N0}
	(D-N_0)\cdot C&\ge 0, \text{with equality unless } \nu=0,\\ \label{eq:N0i}
	(D-N_0)\cdot C_i&=0, \quad i=1,\dots, k, 
\end{align}	
	it is clear that $\nu$ is determined by the claimed intersection numbers.
	Note that $k$ is also determined by the intersection numbers, as the minimum $k$ such that there is a solution of the form $N_0=\nu C+a_1 C_1+ \dots +a_k C_k$ to \eqref{eq:N0}, \eqref{eq:N0i} with $(D-N_0)\cdot C_j\ge 0$ for all $j>k$.
	
	After Proposition \ref{pro:slopes}, the slopes of all sides are determined by the fixed data, thus it is enough to prove that the values $t_i$ are determined by the intersection numbers $D\cdot C_j$, and $D\cdot C$.	
	Let us prove this by induction on $i$. 
	For $i\le k$, $C_i$ belongs to $N_\nu$, so $t_i=\nu$, and we already showed that $k$ is determined by the intersection numbers. 
	So assume $t_i>t_{i-1}\ge \nu$ and $t_1,\dots, t_{i-1}$ are determined by $D\cdot C_j$, and $D\cdot C$.
	Then for $t_{i-1}<t\ll t_{i-1}+1$, 
	$$
	N_{t}=\sum_{j=1}^{i-1} (a_{j0}+a_{j1}t) C_j, 
	$$
	and $t_i$ is the infimum of the $t$ such that 	
	$$
	D-tC-\sum_{j=1}^{i-1} (a_{j0}+a_{j1}t) C_j
	$$
	intersect some $C_{j'}$, $j'\ge i$ negatively.
\end{proof}

\begin{remark}
	Note that the methods of section \ref{sec:proofs} provide information on the set $\mathcal{D}_C(C_1,\dots, C_n)$. For example, if $C_1+\dots+C_n$ is a maximal negative definite configuration and all the $t_i$ are distinct, then this set is an open subset of the big cone.
\end{remark}

{\footnotesize
	\bibliographystyle{plainurl}
\bibliography{NOB}}

	Joaquim Ro\'{e},
	Universitat Aut\`{o}noma de Barcelona, Departament de Matem\`{a}tiques,
08193 Bellaterra (Barcelona), Spain. \\
\nopagebreak
\textit{E-mail address:} \texttt{jroe@mat.uab.cat}

   Tomasz Szemberg, 
   Department of Mathematics, Pedagogical University Cracow,
   Podchor\c{a}\.zych 2, PL-30-084 Krak\'ow, Poland.\\
\nopagebreak
\textit{E-mail address:} \texttt{tomasz.szemberg@gmail.com}

\end{document}